\documentclass{article}
\usepackage{amsmath}    % 公式环境
\usepackage{amsthm}     % 定理环境
\usepackage{amssymb}    % 数学符号
\usepackage{mathtools}  % 数学工具（扩展amsmath）
\usepackage{graphicx}
\usepackage[utf8]{inputenc} % 中文/特殊字符编码
\usepackage{enumitem}  % 列表环境增强
\usepackage{color}     % 颜色支持
\usepackage{mathrsfs}   % 花体数学符号
\usepackage{hyperref}  % 超链接（放最后加载）
\newtheorem{theorem}{Theorem}[section]
\newtheorem{corollary}[theorem]{Corollary}
\newtheorem{lemma}[theorem]{Lemma}
\newtheorem{proposition}[theorem]{Proposition}

\theoremstyle{definition}

\newtheorem{remark}[theorem]{Remark}

\numberwithin{equation}{section} % 公式按section编号
\newcommand\R{\mathbb{R}}
\newcommand\N{\mathbb{N}}
\newcommand\e{\varepsilon}

\newcommand\T{\mathbb{T}}

\DeclareMathOperator{\Supp }{supp}
\title{Norm Inflation for the Critical SQG Equation}
\author{Dengjun Guo\thanks{Academy of Mathematics and Systems Science, Chinese Academy of Sciences; Email: djguo@amss.ac.cn}  \and  Xiaoyutao Luo\thanks{Academy of Mathematics and Systems Science, Chinese Academy of Sciences; Email: xiaoyutao.luo@amss.ac.cn} }

\date{\today}
\begin{document}

\maketitle
\begin{abstract}
We consider the critical dissipative surface quasi-geostrophic (SQG) equation on $\R^2$ or $\T^2$. Despite global regularity of the equation, we show that the data-to-solution map at the critical level $H^1$ is not uniformly bounded. We construct solutions that experience $H^1$ norm inflation from smooth, compactly supported initial data with \emph{large} $H^1$ norm. We also demonstrate  small-data norm inflation in supercritical Sobolev spaces $W^{\beta,p}$ for $1<p<2$ and $1\le\beta<\tfrac{2}{p}$. 
\end{abstract}

$\mathbf{Key\ words:}$ Norm inflation; Critical SQG equations; Supercritical Sobolev spaces

%%%%%%%%%%%%%%%%%%%%%%%%%%%%%%%%%%%%%%%%%%%%%%%%%%%%%%%%%%%%%%%%%%%%%%%%%%%%%%%%%%%%
\section{Introduction}
%%%%%%%%%%%%%%%%%%%%%%%%%%%%%%%%%%%%%%%%%%%%%%%%%%%%%%%%%%%%%%%%%%%%%%%%%%%%%%%%%%%%
In this paper, we consider the critical surface quasi-geostrophic (SQG) equation
\begin{equation}\label{eq sqg}\left\{
\begin{aligned}
	&\partial_t \theta + \Lambda \theta + u \cdot \nabla \theta  = 0, \\
	&\theta(x,0) = \theta_0(x),
\end{aligned}\right.
\end{equation}
where the velocity field $u$ is given by
$$
u = u[\theta] := \nabla^{\perp} \Lambda^{-1} \theta,
$$
and $\Lambda = (-\Delta)^{1/2}$. The domain is either $\Omega = \R^2 $ or $\mathbb{T}^2$. The equation admits the scaling
$$
\theta_\lambda(t,x) = \theta(\lambda t, \lambda x).
$$
Function spaces that are invariant under this scaling are called \emph{critical} for~\eqref{eq sqg}; for instance, $H^1$ and $W^{\frac2p,p}$ are critical spaces. It is well known that~\eqref{eq sqg} is globally well-posed in $H^1(\Omega)$.

Our first main result shows that, even though the critical SQG equation is globally well-posed in $H^1$, the data-to-solution map fails to be uniformly bounded.

\begin{theorem}[Large-data norm inflation in $H^1$]\label{thm:critical}
Let $\Omega = \R^2$ or $\mathbb{T}^2$. Then for any $\e > 0$, there exist a time $0<T \le \e$ and a  solution $\theta  $ of~\eqref{eq sqg} on $[0,T]$    such that  
\begin{enumerate}
	\item The initial data satisfy $\theta_0 \in C^\infty_c(\R^2) $ and
\begin{equation}\label{eq thm:critical 1}
 \| \theta_0 \|_{H^{1-\e}} \leq \e
\end{equation}
	\item The solution $\theta$ satisfies
	\begin{equation}\label{eq thm:critical 2}
	\frac{\|\theta(T)\|_{H^1}}{\|\theta_0\|_{H^1}} \ge \varepsilon^{-1}.
	\end{equation}
\end{enumerate}
\end{theorem}
 
\begin{remark}
The initial data in Theorem~\ref{thm:critical} are \emph{large} in $H^1$, which is necessary and sharp. Indeed, by the small-data global well-posedness result of \cite{Miura2006}, if the initial data $ \|\theta_0\|_{H^1} \leq \e$,  then the global solution $\|\theta(t)\|_{H^1} \leq C \e  $ for all times $t\ge 0$.

Such a result is surprising because global well-posedness in $H^1$ of~\eqref{eq sqg} is known~\cite{Dong2010,DongDu2008} even for large data. Our result shows a failure of uniform boundedness of the data-to-solution map $\theta_0 \mapsto \theta(t)$ in the critical topology $H^1$, in contrast to the maximum principle $
\|\theta(t) \|_{L^\infty} \leq \|\theta_0 \|_{L^\infty} $
despite the fact that both $\dot H^1(\R^2)$ and $L^\infty(\R^2)$ have the same scaling property.  
 
\end{remark}

 \begin{remark}
For any $s >1$, we can construct a solution to \eqref{eq sqg} such that Large-data norm inflation holds in $H^s$: $\frac{\|\theta(T)\|_{H^s}}{\|\theta_0\|_{H^s}} \ge \varepsilon^{-1}.$ 

In other words, the failure of the boundedness of the data-to-solution map is not restricted to the critical space $H^1$ and holds for a large range of spaces.
\end{remark}

Theorem \ref{thm:critical} is in fact a corollary of a more ``traditional'' small-data norm inflation result in a supercritical regime.

\begin{theorem}[Small-data norm inflation]\label{thm:supercritical}
Let $\Omega = \R^2 $ or $\mathbb{T}^2$. Fix $1<p<2$ and $1 \le \beta < \tfrac{2}{p}$.  For any $\varepsilon>0$, there exists a smooth solution $\theta$ of~\eqref{eq sqg}, and a time $0<T \le \varepsilon$ such that  
\begin{enumerate}
    \item The initial data satisfy $\theta_0 \in C^\infty_c(\R^2) $ and  
    \[
    \|\theta_0\|_{W^{\beta,p}} \lesssim \varepsilon.
    \]
	\item The solution $\theta$ exhibits norm inflation at $t=T$:
	\[
	 \|\theta(T)\|_{W^{\beta,p}} \gtrsim \varepsilon^{-1}.
	\]
\end{enumerate}
\end{theorem}

\begin{remark}
In contrast to Theorem~\ref{thm:critical}, here the initial data can be taken \emph{arbitrarily small} in the supercritical norm. This shows a genuine ill-posedness feature in supercritical spaces $W^{\beta,p}$ with $1<p<2$, where boundedness of the data-to-solution map is  unexpected.  
\end{remark}

\begin{remark}
Our construction is based on the idea of constructing multi-scale approximate solutions, pioneered in \cite{CordobaMartinez-Zoroa2024} and developed in \cite{MR4768540,MR4776418,CordobaMartinezZoronaOzanski2024}. Since we are treating the dissipation term as an error, the main challenge is the smoothing of the critical dissipation term. 
 
\end{remark}

\begin{remark}
	The proofs for the cases $\R^2$ and $\T^2$ are essentially the same. We construct the approximate solution $\overline{\theta}$ to the critical SQG equation on $\R^2$; the periodic case then follows by taking the periodic extension of $\overline{\theta}$.  
\end{remark}

\subsection{Background}

The critical dissipative SQG equation belongs to the family of dissipative active scalar equations
\begin{equation}\label{eq:alpha-sqg}
\left\{
\begin{aligned}
&\partial_t \theta + \Lambda^\alpha \theta + u\cdot\nabla\theta  = 0,\\
&\theta(x,0)=\theta_0(x),
\end{aligned}\right.
\end{equation}
where $u=\nabla^\perp\Lambda^{-1}\theta$ and $\Lambda^\alpha$ is the fractional Laplacian with Fourier multiplier $|\xi|^\alpha$ for any $\alpha \in (0,2)$.

The equation \eqref{eq:alpha-sqg} is invariant under the natural scaling 
\begin{equation}\label{eq:alpha-sqg-scaling}
\theta \mapsto \theta_\lambda(x,t):=\lambda^{\alpha-1}\theta(\lambda x,\lambda^\alpha t).
\end{equation}

Since the scaling-invariant (critical) functional spaces depend on $\alpha$, by the maximum principle $\| \theta \|_{L^\infty} \leq \| \theta_0 \|_{L^\infty}  $  and the scaling \eqref{eq:alpha-sqg-scaling}, we can then classify the equations into three cases:

\begin{itemize}
    \item \textbf{Subcritical regime} ($\alpha>1$). The dissipation is stronger than the critical scaling and provides significant smoothing; global regularity and stability results are comparatively well understood in this regime. Classical references and overviews include \cite{CarrilloFerreira2008,ConstantinWu1999,DongLi2008,Wu2001} and related works treating dissipation-dominated regimes.

    \item \textbf{Critical case} ($\alpha=1$). The balance between transport and dissipation is subtle. The breakthrough of global well-posedness was established by Caffarelli–Vasseur~\cite{CaffarelliVasseur2010} and independently by Kiselev–Nazarov–Volberg~\cite{KiselevNazarovVolberg2007}.  Dong~\cite{Dong2010} proved global well-posedness in the periodic homogeneous Sobolev space $\dot H^1(\mathbb{T}^2)$, and Dong and Du~\cite{DongDu2008} established global well-posedness and decay estimates for arbitrary $H^1(\mathbb{R}^2)$ data. Early works include the small–$L^\infty$ global theory by Constantin–Córdoba–Wu~\cite{ConstantinCordobaWu2001} and the systematic study in the thesis of Resnick~\cite{Resnick1995}.

    \item \textbf{Supercritical regime} ($0<\alpha<1$ and the inviscid case). Dissipation is weaker than the nonlinear transport, and the problem becomes substantially harder. Global regularity remains open in general; however, there are partial results on eventual regularization \cite{Dabkowski2011,Silvestre2010}, partial regularity~\cite{MR4235801}, local well-posedness and small-data global well-posedness \cite{Ju2007,Miura2006}, global well-posedness for slightly supercritical SQG equations \cite{CotiZelatiVicol2016,DabkowskiKiselevSilvestreVicol2014,DabkowskiKiselevVicol2012,XueZheng2012} and references therein. The inviscid SQG equation is also intensively studied~\cite{MR3987721,MR1304437,MR2357424}.
\end{itemize}

%%%%%%%%%%%%%%%%%%%%%%%%%%%%%%%%%%%%%%%%%%%%%%%%%%%%%
\paragraph{Related norm inflation constructions.}
Recent progress has considerably deepened our understanding of the norm inflation phenomenon since the seminal work of Bourgain and Li~\cite{bourgainli} in the context of Euler equations.

\medskip
\noindent\textbf{(1) Inviscid SQG equations.}
For the inviscid SQG equation, C\'ordoba, Mart\'inez-Zoroa~\cite{MR4452676} and simultaneously and independently Jeong and Kim~\cite{JeongKimAPDE2024}  proved \emph{norm inflation and strong ill-posedness} in the critical Sobolev space $H^2$. The work \cite{MR4452676} also demonstrates the same behavior in a supercritical regime $H^s$, $\frac{3}{2}<s<2$. Later, C\'ordoba, Mart\'inez-Zoroa, and Ożański \cite{CordobaMartinezZoronaOzanski2024} obtained examples of \emph{instantaneous continuous loss of regularity} for the inviscid SQG equation.

The works~\cite{arXiv:2502.06357,MR4768540}  extended these phenomena to generalized inviscid SQG models, proving norm inflation and ill-posedness across a range of supercritical Sobolev spaces.  

\medskip
\noindent\textbf{(2) Viscous SQG equations.}

 In the viscous case,  norm inflation constructions have only been demonstrated in the supercritical regime $\alpha<1$.

In~\cite{CordobaMartinez-Zoroa2024}, C\'ordoba and Mart\'inez-Zoroa considered the dissipative fractional SQG equation with $0<\alpha<1$, and constructed global unique solutions that nevertheless exhibit \emph{instantaneous loss of Sobolev regularity}.  Their method relies on a high-frequency ansatz and delicate control of nonlinear interactions.

Our work follows the general framework of C\'ordoba and Mart\'inez-Zoroa~\cite{MR4452676,CordobaMartinez-Zoroa2024}, but the critical case $\alpha=1$ presents additional challenges: the dissipation is stronger and its smoothing effect must be treated as an error at critical regularity. Overcoming this difficulty is a core component of our construction.

\medskip

To our knowledge, Theorem \ref{thm:critical} represents the first example of a short-time norm inflation in a space where it is simultaneously globally well-posed. In the spaces of local well-posedness, previous studies have only shown non-uniform continuity~\cite{MR2606636,MR4031801}. For more related works, we refer to \cite{ChoiJungKim2025,MR4776418,ElgindiMasmoudi2019,jeong2025instantaneouscontinuouslosssobolev,KimJeong2022, arXiv:2404.07813,arXiv:2504.08288}.

\subsection{Idea of the proof}
We briefly sketch the main idea behind the norm inflation construction on $\R^2$. In a nutshell, controlling stronger dissipation requires more spatial concentration, thus our initial data is only small in $W^{\beta,p}$ for $p<2$.

Our approach, following the framework of \cite{MR4452676, CordobaMartinez-Zoroa2024}, is to construct an approximate solution $\overline{\theta}$ to \eqref{eq sqg} and control the remainder $\eta$ in the decomposition  
\[
\theta  = \overline{\theta}      +  \eta .
\]

The approximate solution (termed ``pseudo-solution'' in \cite{MR4452676, CordobaMartinez-Zoroa2024}) in polar coordinate $(r,\alpha)$ is given by
\begin{equation}\label{eq:intro_appro}
	\overline{\theta} =\underbrace{\e \lambda^{\frac{2}{p}-\beta}  f(\lambda r)}_{=:\overline{\theta}_1, \text{ stationary and radial}} + \underbrace{\e\lambda^{\frac{2}{p}-\beta}N^{-\beta}g(\lambda r)\cos\left( N\alpha-\e Nt\lambda^{2-\beta} h(\lambda r) \right)}_{=:\overline{\theta}_2, \text{ transported by the velocity of $\overline{\theta}_1$}}  .
\end{equation}
Here $f,g$ are smooth compactly supported profiles and $h$ is the angular velocity of the profile $f$. The main parameter $\lambda>0$ will be taken to be sufficiently large in the proof, depending on the given $\e$ and a few universal constants. Notice the scaling  $\lambda^{\frac{2}{p}-\beta}$ ensures the normalization of $\dot W^{\beta,p}$ norm.

The idea behind \eqref{eq:intro_appro} is rather simple: $\overline{\theta}_1$ produces ``larger'' velocity than $\overline{\theta}_2$ due to the oscillation in $\overline{\theta}_2$ (see Theorem \ref{thm:approx velocity}), so the solution is essentially transported by the velocity of $\overline{\theta}_1$.

It is then not hard to show the approximate solution \eqref{eq:intro_appro} satisfies for some time \(  T = \lambda^{\beta-1-\frac2p}\epsilon^{-1-\frac{2}{\beta}}>0  \) the small-data norm inflation in $W^{\beta ,p }$:
\[
\|\overline{\theta} (0)\|_{W^{\beta, p }} \sim \varepsilon, \quad \|\overline{\theta} (T)\|_{W^{\beta, p }} \sim \varepsilon^{-1}.
\]

Remarkably, when taken $\beta=1$, this approximate solution also exhibits norm inflation in \( H^1 \) (note that the initial $H^1$ norm is very large): 
\begin{equation*}
\begin{cases}
\|\overline{\theta} (T)\|_{H^1} \gtrsim \e^{-1} \lambda^{-\beta + \frac{2}{p} }  &\\
\|\overline{\theta} (0)\|_{H^1}  \lesssim \e \lambda^{-\beta + \frac{2}{p} } &
\end{cases}
\quad \text{and hence } \quad   \frac{\|\overline{\theta} (T)\|_{H^1}}{\| \overline{\theta} (0)\|_{H^1}} \ge \e^{-1} .
\end{equation*}

To finish the proof, it is important that the same norm inflation can be shown for the exact solution with the same initial data $ \overline{\theta}(0)$. However, the approximation ansatz  is only valid in a certain regime. In the inviscid case or viscous case with supercritical diffusion, this approximation has been justified in \cite{MR4452676,CordobaMartinez-Zoroa2024}. However, in our critical settings, we have the following two constraints:

\begin{itemize}
\item  \textbf{The smallness of $\eta$} requires $\beta \ge  1$. The approximation $\theta \approx \overline{\theta}$ remains valid only if the self-interaction of $\overline{\theta}_2$ does not generate a dominant velocity gradient; otherwise,  the velocity from $\overline{\theta}_2$ disrupts the flow, and our ansatz ceases to be a good approximation.

\item  \textbf{Neglecting the dissipation}  requires $ \frac{2}{p} >\beta  $. For norm inflation to occur, the nonlinear growth must dominate the smoothing effect of the dissipation $\Lambda \theta$; otherwise, linear smoothing dominates, and no growth can be observed.

\end{itemize}

These two constraints requires us to work in supercritical $W^{\beta ,p}$ spaces with $p<2$ as these spaces impose stronger spatial concentration than $H^{\beta+ 1 - \frac{2}{p} }$ despite having the same scaling. These technical obstacles limit us from showing the small-data norm inflation in supercritical $H^{\beta  }$  for $\beta<1$ which remains an important open question for the critical SQG equation.

\subsection*{Outline of the paper.} The rest of the paper is structured as follows:
\begin{itemize}
    \item Section 2 introduces the preliminaries of the necessary functional spaces. 
    \item Section 3 constructs the approximate solution and estimates its approximation error to \eqref{eq sqg}.

    \item Section 4 compares approximate and exact solutions, establishing energy estimate up to the critical time $T$, then combines these results to obtain norm inflation in supercritical space $W^{\beta,p}$ and critical space $H^1$. 
    
    \item Section 5 briefly outlines the proof in periodic setting.
 
    \item Appendix collects the proofs of several useful lemma used in the construction.
\end{itemize}

%%%%%%%%%%%%%%%%%%%%%%%%%%%%%%%%%%%%%%%%%%%%%%%%%%%%%%%%%%%%%%%%%%%%%%%%%%%%%%%%%%%%  
\section{Preliminaries}
%%%%%%%%%%%%%%%%%%%%%%%%%%%%%%%%%%%%%%%%%%%%%%%%%%%%%%%%%%%%%%%%%%%%%%%%%%%%%%%%%%%%

%%%%%%%%%%%%%%%%%%%%%%%%%%%%%%%%%%%%%%%%%%
\subsection{Notations}
%%%%%%%%%%%%%%%%%%%%%%%%%%%%%%%%%%%%%%%%%%
 For a vector- or tensor-valued function \( f \), its modulus \( |f| \) denotes the square root of the sum of squares of each component. For any real number \( s \), we denote by $\lfloor s \rfloor$ the greatest integer less than or equal to $s$ while $\langle s \rangle:=\sqrt{1+ s^2}\,$ is the standard Japanese bracket.

For a Banach space \( X \), its norm is denoted by \( \|\cdot\|_X \). Functional norms in this paper are mostly defined on \( \mathbb{R}^2 \), so we often write \( \|f\|_{L^p} \) and \( \|f\|_{L^\infty} \) for brevity.

For two quantities \( X, Y \ge 0 \), we write \( X \lesssim Y \) if \( X \le CY \) holds for some constant \( C > 0 \), and similarly \( X \gtrsim Y \) if \( X \ge CY \), and \( X \sim Y \) means \( X \lesssim Y \) and \( X \gtrsim Y \) at the same time. In addition, \( X \lesssim_{a,b,c,\dots} Y \) means \( X \le C_{a,b,c,\dots} Y \) for a constant \( C_{a,b,c,\dots} \) depending on parameters \( a, b, c, \dots \).  

Throughout the paper, for \( k \in \mathbb{N} \), \( \nabla^k \) refers to the full gradient in \( \mathbb{R}^2 \).  We will also use the polar coordinate $x=(r\cos\alpha,r\sin\alpha)$ and  $\vec{e}_r$ and $\vec{e}_{\alpha}$ denote $\frac{x}{|x|}$ and $\frac{x^{\perp}}{|x|}$, respectively.

\subsection{Sobolev spaces}

We recall the definition of Sobolev spaces for $1 \le p \le \infty$ and integer \( k \in \mathbb{N} \):  
\begin{equation}\label{eq sec1 A2}
\|f\|_{W^{k,p}} = \sum_{0\le|i| \le k} \|\nabla^i f\|_{L^p}
\end{equation}  
and
\begin{equation}\label{eq sec1 A3}
\|f\|_{\dot{W}^{k,p}} = \|\nabla^k f\|_{L^p}.
\end{equation}  

We also recall the fractional Sobolev spaces of the following definition. For real \( s \in \mathbb{R} \) and \( 1 < p < \infty \), 
\begin{equation}
\|f\|_{W^{s,p}} = \|J^s f\|_{L^p}
\end{equation} 
and
\begin{equation}\label{eq sec1 A33}
 \|f\|_{\dot{W}^{s,p}} = \|\Lambda^s f\|_{L^p},
\end{equation}  
where \( J^s \) is the Bessel potential with Fourier multiplier \( \widehat{J^s f}(\xi) = (1 + |\xi|^2)^{s/2} \hat{f}(\xi) \), and \( \Lambda^s \) is the Riesz potential with Fourier multiplier \( \widehat{\Lambda^s f}(\xi) = |\xi|^s \hat{f}(\xi) \). When \( p = 2 \), we denote \( H^s := W^{s,2} \) and $\dot{H}^s:=\dot{W}^{s,2}$.

It is well known that for \( 1 < p < \infty \), the two definitions \eqref{eq sec1 A3} and \eqref{eq sec1 A33} coincide. We emphasize that when \( p = \infty \), we only use the definition \eqref{eq sec1 A3}.

\subsection{Besov spaces}

We recall the definition of Besov spaces by Littlewood-Paley decomposition. Since we only use some duality and interpolation properties of Besov spaces, we refer to \cite{BCD} for the details.

Let \( \{\Delta_q\}_{q \in \mathbb{Z}} \) be a sequence of Littlewood-Paley projection operators such that \( \widehat{\Delta_q} \) is supported in frequencies \( |\xi| \sim 2^q \), and \( \text{Id} = \sum_{q \in \mathbb{Z}} \Delta_q \) in the sense of distributions.

For \( s \in \mathbb{R} \), \( 1 \le p, q \le \infty \), the homogeneous Besov norm \( \|\cdot\|_{\dot{B}^s_{p,q}} \) is defined by:  
\begin{equation}
\|f\|_{\dot{B}^s_{p,q}} = \left\| \, 2^{qs} \|\Delta_q f\|_{L^p} \, \right\|_{\ell^q(\mathbb{Z})}
\end{equation}  
(where \( \ell^\infty \) is used for \( q = \infty \), i.e., the supremum over \( q \)).

In what follows, we will frequently use the following standard interpolation result. We include a proof in the appendix for the reader's convenience.
\begin{lemma}\label{le interpolation}
	For any $s \in \mathbb{R}$, $1< p \le \infty$, $\delta_1<0$ and $\delta_2>0$, we have
	$$\|f\|_{\dot{W}^{s,\infty}}\lesssim \|f\|_{\dot{W}^{s+\delta_1+\frac{2}{p},p}}^{\theta}\|f\|_{\dot{W}^{s+\delta_2+\frac{2}{p},p}}^{1-\theta},$$
	where $\theta=\frac{\delta_2}{\delta_2-\delta_1}$.
\end{lemma}

\subsection{Local approximation for $\Lambda^{-s}$}

We will use the following estimates for the fractional operator \(\Lambda^{-s}\) acting on oscillatory functions of the form  \[\theta(x) = g(r)\cos(N\alpha + Nh(r)),\]   since the approximate solution exhibits this structure. 

These estimates, first proved in \cite{CordobaMartinez-Zoroa2024} for $0<s \le 1$ and $p=2$, will play a crucial role in the proof of our main theorem in subsequent sections.

Here we present a proof valid for all $s >0$. 
\begin{theorem}\label{thm:approx velocity}
Let $g, h\in C^\infty(\R) $ with $ \operatorname{supp} g \subset (0,\infty) $. For $N \in \N$, we define 
\[
	\theta(x) := g(r) \cos\bigl(N\phi(x) \bigr),
\]
where $ \phi(x) = \alpha + h(r)$ in polar coordinates.     

For any $  s>0 $, there exists  a dimensional constant $c_{s}>0$ such that  for any $k>0$ and $1 < p < \infty$, or $k\in \N$ and $p=\infty$, the following holds:
	 \begin{equation}\label{eq Lambda Lp} 
	\left\|  \Lambda^{-s} \theta   - c_s N^{-s}\frac{\theta  }{\left (\frac{1}{r^2}  +  |h'(r)|^2\right)^\frac{s}{2} }    \right\|_{\dot{W}^{k,p} } \lesssim_{s,k,h,g} N^{k-1-s} ,
	\end{equation} 
where the implicit constant depends on $s, k, \operatorname{supp} g$, \( \|  h \|_{C^n} \) and \( \|  g \|_{C^n} \) for some \( n  =n(s,k)\). 
\end{theorem}

For later use, we record the rescaled version of the above estimates. Define $\theta_{\lambda}(x)=\theta(\lambda x)$. Using the identity $$\Lambda^{-s}(\theta_{\lambda}( x))=\lambda^{-s}\left( \Lambda^{-s}\theta \right)_{\lambda}(x),$$ we deduce the following corollary directly from \eqref{eq Lambda Lp}:
\begin{corollary} \label{cor:approx velocity}
Under the setting of Theorem \ref{thm:approx velocity}, the rescaled objects satisfy
    
	 \begin{equation}\label{eq lambda Lp} 
	\left\| \Lambda^{-s} \theta_{\lambda}  - c_s (\lambda N)^{-s}\frac{\theta_{\lambda}  }{\left (\frac{1}{(\lambda r)^2}  +  |h'(\lambda r)|^2\right)^\frac{s}{2} }   \right\|_{\dot{W}^{k,p}} \lesssim_{k, \gamma, h, g} \lambda^{k-s-\frac{2}{p}}N^{k-1-s}.
	\end{equation}    
\end{corollary}

  Indeed, \eqref{eq Lambda Lp} is a direct consequence of the following lemma, whose proof is provided in the Appendix.

\begin{lemma}\label{le osc}
Under the setting of Theorem \ref{thm:approx velocity}, let $\gamma>0$  be such that $ \operatorname{supp} g \subset  [\frac{1}{\gamma},\gamma]$. Then for any $k\in \N$, the following pointwise bounds hold:
\begin{align}\label{eq Lambda out}
	\left|  \nabla^k \left(\Lambda^{-s} \theta(x)  - c_s N^{-s}\frac{\theta  }{\left(\frac{1}{r^2}  +  |h'(r)|^2\right)^\frac{s}{2} }\right)  \right|  \lesssim
\begin{cases}
 N^{k-1-s  }  \quad & \text{if $|x|\le 2\gamma$}   \\
N^{k-1-s  }r^{-3} \quad &\text{if $|x|\ge 2\gamma$}, 
\end{cases}
\end{align}
where the implicit constant depends on $s, k, \gamma$, \( \|  h \|_{C^n} \), and \( \|  g \|_{C^n} \) for some \( n  =n(s,k)\).

\end{lemma}

%%%%%%%%%%%%%%%%%%%%%%%%%%%%%%%%%%%%%%%%%%%%%%%%%%%%%%%%%%%%%%%%%%%%%%%%%%%%%%%%%%%%
\section{Proof of Theorem \ref{thm:critical}: the approximate solution}
%%%%%%%%%%%%%%%%%%%%%%%%%%%%%%%%%%%%%%%%%%%%%%%%%%%%%%%%%%%%%%%%%%%%%%%%%%%%%%%%%%%%

In this section, we construct the approximate solution that exhibits norm inflation in \( W^{\beta,p} \) for $1<p<2$ and $1\le \beta < \frac2p$.

\subsection{Choice of the parameter}
We first fix the choice of parameters appearing in the construction. Given $\varepsilon>0$ as in Theorem \ref{thm:critical}, fix $1<p<2$ and $1\le \beta < \frac2p$. Then for any sufficiently large $N \in \N$ we set
\begin{align}
        \lambda & =  N^{\frac{100p}{2-p\beta}}, \label{eq:parameters lambda}\\
        T & = \lambda^{\beta-1-\frac2p}\e^{-1-\frac{2}{\beta}} \label{eq:parameters T}.
\end{align}
In particular, it follows that
\begin{equation}
        \lambda^{\beta-\frac2p} =N^{-100}.
\end{equation}

These choices are found by solving a few constraints in the course of the proof. In particular, \eqref{eq:parameters lambda} is used for controlling the dissipation while \eqref{eq:parameters T} is for both showing the norm inflation and controlling the remainder.

The value of $N$ is assumed to be very large and will be fixed in the end of the proof, depending on $\e$, $p$ and $\beta$.  
\subsection{Definition of the approximate solution}\label{se sec3 definition 3.1}

Following the construction of pseudo-solution in \cite{MR4452676, CordobaMartinez-Zoroa2024}, we define the approximate solution $\overline{\theta}:\R^2 \times [0, \infty) \to \R$ by
\begin{equation}\label{eq:def_theta bar}
	\overline{\theta} = \overline{\theta}_1 + \overline{\theta}_2,
\end{equation}
where $\overline{\theta}_1$ and $\overline{\theta}_2$ are defined in polar coordinates $(r,\alpha)$ by
\begin{equation}\label{eq:def_theta bar 1}
	\overline{\theta}_1=\e\lambda^{\frac2p-\beta}f(\lambda r),
\end{equation}
\begin{equation}\label{eq:def_theta bar 2}
	\overline{\theta}_2=\e\lambda^{\frac2p-\beta}N^{-\beta}g(\lambda r)\cos\left( N\alpha-\e Nt\lambda^{1+\frac2p-\beta} h(\lambda r) \right),
\end{equation}
where $N \in \N$ is a large parameter that we fix at the end of the proof and the profiles $f,g$ and $h$ are chosen as follows.
\begin{itemize}
	\item First, we require the profile $f \in C^\infty(\R)$   such that $\operatorname{supp}(f) \subset [\frac12, 2]$.

\item  Next, we require the function $h$, defined by  
\begin{equation}\label{eq:def_h}
h(r)=\frac{u[f](x)}{r} \cdot \vec{e}_{\alpha} 
\end{equation}
satisfies  $h'(r_f) >0 $ at some point $r_f > 2$. The existence of such a $f$ is proved in Lemma \ref{le fgh}.
	
	\item Finally, we fix the function $g \in C_c^\infty(\R)$ such that $ \operatorname{supp}(f) \cap \operatorname{supp}(g) = \emptyset $ and $ h'(r) >0 $ on $ \operatorname{supp}(g)$.

\end{itemize}

Since $\overline{\theta}_1$ radial, by \eqref{eq:def_h} we see that $\overline{\theta}_2$ is transported by the angular velocity of $\overline{\theta}_1 $ and satisfies the equation
$$
\partial_t \overline{\theta}_2+u[\overline{\theta}_1] \cdot \nabla \overline{\theta}_2=0.
$$

Therefore, the approximate solution $\overline{\theta} =  \overline{\theta}_1 + \overline{\theta}_2$ obeys the approximate equation
\begin{equation}
\begin{cases}
\partial_t \overline{\theta} +  \Lambda \overline{\theta} +u[\overline{\theta} ] \cdot \nabla \overline{\theta} = \overline{F} ,& \\   
 u[\theta]  = \nabla^{\perp} \Lambda^{-1}   \theta, &
\end{cases}
\end{equation}
where the error term $\overline{F}$ is given by
\begin{equation}\begin{aligned}
	\overline{F}&=u[\overline{\theta}_2]\cdot \nabla \overline{\theta}_1+u[{\overline{\theta}_2}]\cdot\nabla \overline{\theta}_2+\Lambda \overline{\theta}_1+\Lambda \overline{\theta}_2 \\
	&=: \overline{F}_1+\overline{F}_2+\overline{F}_3+\overline{F}_4.
\end{aligned}\end{equation}

We will prove that $\overline{\theta}$ develops the norm inflation at some time $T$. For now, we first derive estimates for $\overline{\theta}$, then show that it approximates the full equation \eqref{eq sqg} with a small error term $\overline{F}$.

\subsubsection{Heuristics of the approximation}
Next, we will illustrate the construction of the approximate solution and why the parameter is chosen in such a way. Following the construction above, a formal computation gives 
$$
\| \nabla^{\beta}\overline{\theta} (0)\|_{L^p} \lesssim \e \quad \text{and}\quad  \| \nabla^{\beta}\overline{\theta} (t)\|_{L^p} \gtrsim  \,  \e^{1+\beta} \left( t \lambda^{1+\frac2p-\beta} \right)^{\beta}.
$$
Therefore, to ensure norm inflation at time $T$, we set $T=\lambda^{\beta-1-\frac2p}\e^{-1-\frac{2}{\beta}}$.

Now for $t \le T$, consider the error term
$$\overline{F}  =u[\overline{\theta}_2]\cdot \nabla \overline{\theta}_1+u[\overline{\theta}_2]\cdot \nabla \overline{\theta}_2+\Lambda \overline{\theta}.$$
Thanks to Theorem \ref{thm:approx velocity}, both the leading terms of $u[\overline{\theta}_2]\cdot \nabla \overline{\theta}_2$ and $u[\overline{\theta}_2]\cdot \nabla \overline{\theta}_1$ vanish and present no constraints in the choice of parameters.

To overcome the smoothing of the dissipation, we need $T |\Lambda \overline{\theta}| \ll| \overline{\theta}| $, namely
$$
\lambda^{-\frac2p+\beta}N^{1 } \ll_\e  1 ,
$$
which require $\beta<\frac2p$ and $N \ll \lambda$.

Taking the difference between the exact and approximate solutions, the remainder term $\eta $  formally  satisfies 
$$
 \frac{|\eta| }{|\overline{\theta}_2|}  \lesssim  \frac{Te^{T\|\nabla u[\overline{\theta}]\|_{L^{\infty}}}|\overline{F}| }{|\overline{\theta}_2|}  \lesssim_{\e}  N^{-1 },
$$
where we have assumed that 
$$T\|\nabla u[\overline{\theta}]\|_{L^{\infty}} \sim T\| \nabla \overline{\theta}\|_{L^{\infty}}\lesssim \e^{-\frac{2}{\beta}} \left(1+N^{1-\beta}\right).$$ 
Since $N^{1-\beta}$ appears in the exponent of the exponential, we must assume $ \beta \ge 1$.

In the rest of this section, we rigorously justify the above formal computation, and the smallness of $\eta$ will be established in the next section.

\subsection{Estimates for $\overline{\theta}$} 
We begin by estimating the approximate solution    $\overline{\theta}=\overline{\theta}_1+\overline{\theta}_2$ and its associated velocity $ u[\overline{\theta}]$.

Recall that
$$
 \overline{\theta}_2 = \e \lambda^{\frac2p-\beta}N^{-\beta} g(  \lambda r)\cos\left( N \alpha- N h_2( t, \lambda r) ) \right)
$$
with $ 
h_2(t, r) =  \e  t\lambda^{1+\frac2p-\beta} h(  r)$ .
Motivated by Corollary \ref{cor:approx velocity}, we define the approximate velocity $\overline{u}[\overline{\theta}_2]$ by
\begin{equation}\label{eq:appro_u}
\overline{u}[\overline{\theta}_2]:=\nabla^{\perp}\overline{\Lambda}^{-1}\overline{\theta}_2,
\end{equation}
where (more generally, for any $s>0$) we adopt the notation
\begin{equation}\label{eq sec3 psi2}
    \overline{\Lambda}^{-s}\overline{\theta}_2:= (N\lambda)^{-s} \frac{\overline{\theta}_2 }{ \psi_{2,s}(t, \lambda r)}
\end{equation}
and $
\psi_{2,s}(t, r) : =  c_s^{-1}\left(    r^{-2}    + | h'_2(  r) | ^2   \right)^{\frac{s}{2}}$. 

\begin{lemma}\label{le sec3 approximate velocity}
    For any $k\ge 0$ and $1< q < \infty$, or $k\in \N$ and $q =\infty$, if $0\leq t\leq T$, there holds 
    \begin{equation}\label{eq sec3 approximate velocity3}
        \| \overline{u}[\overline{\theta}_2]\|_{\dot{W}^{k,q}} \lesssim_{\e,k,q} (\lambda N)^{k} \lambda^{\frac2p-\frac2q-\beta}N^{-\beta}
    \end{equation}
    and
    \begin{equation}\label{eq sec3 approximate velocity4}
		\|u[\overline{\theta}_2]-\overline{u}[\overline{\theta}_2]\|_{\dot{W}^{k,q}}\lesssim_{\e,k,q}(\lambda N)^{k} \lambda^{\frac2p-\frac2q-\beta}N^{-\beta-1}.
	\end{equation}
\end{lemma}

\begin{proof}
    It suffices to consider $k \in \N$ since the general case then follows from interpolation.

    It also suffices to show
    \begin{equation}\label{eq sec3 approximate velocity1}
        \| \overline{\Lambda}^{-s}\overline{\theta}_2\|_{\dot{W}^{k,q}} \lesssim_{\e,k,q,s} (\lambda N)^{k-s} \lambda^{\frac2p-\frac2q-\beta}N^{-\beta}
    \end{equation}
    and
    \begin{equation}\label{eq sec3 approximate velocity2}
        \| \Lambda^{-s}\overline{\theta}_2- \overline{\Lambda}^{-s}\overline{\theta}_2\|_{\dot{W}^{k,q}} \lesssim_{\e,k,q,s}(\lambda N)^{k-s} \lambda^{\frac2p-\frac2q-\beta}N^{-\beta-1}.
    \end{equation}
    
    Note that the denominator $\psi_{2,s}(t, \lambda r) $ in \eqref{eq sec3 psi2} is smooth and bounded away from $0$ on the support of $\overline{\theta}_2 $. Precisely, for $x \in \operatorname{Supp} \overline{\theta}_2$, there hold
\begin{align}\label{eq es for F22 2}
\psi_{2,s}(\lambda x) \gtrsim 1 , \quad \text{and}\quad
\frac{|\nabla^k \psi_{2,s}(\lambda x)|}{|\psi_{2,s}(\lambda x)|} \lesssim 1, \quad \text{for any $k\in\N$}. 
\end{align}
Therefore, a direct calculus yields \eqref{eq sec3 approximate velocity1}, and \eqref{eq sec3 approximate velocity2} follows from Corollary \ref{cor:approx velocity}. Setting $k=k'+1$ and $s=1$ in \eqref{eq sec3 approximate velocity1} and\eqref{eq sec3 approximate velocity2} gives \eqref{eq sec3 approximate velocity3} and \eqref{eq sec3 approximate velocity4}.
\end{proof}

Next, we estimate the approximate velocity $u[\overline{\theta}]=u[\overline{\theta}_1]+u[\overline{\theta}_2]$.

\begin{proposition}\label{prop est bar theta}
	For any $k\ge 0$ and $1< q < \infty$, or $k\in \N$ and $q =\infty$, if $0\leq t\leq T$, there holds  
	\begin{equation}\label{eq bound for velocity theta1}
		\|u[\overline{\theta}_1]\|_{\dot{W}^{k,q}}+	\|\overline{\theta}_1\|_{\dot{W}^{k,q}}\lesssim_{\e, k,  q}\lambda^{k}\lambda^{\frac2p-\frac2q-\beta}
	\end{equation}
    and
	\begin{equation}\label{eq bound for velocity}
		\|u[\overline{\theta}_2]\|_{\dot{W}^{k,q}}+	\|\overline{\theta}_2\|_{\dot{W}^{k,q}}\lesssim_{\e, k,  q} \lambda^{\frac2p-\frac2q-\beta}N^{-\beta} (\lambda N)^k  .
	\end{equation} 
\end{proposition}
\begin{proof}
It suffices to consider $k\in \N$, as the general case follows immediately by interpolation.

The estimates for $\overline{\theta}_1$ and $u[\overline{\theta}_1] $ can be handled directly by the scaling factor $\lambda$. The estimates for $u[\overline{\theta}_2]$ follow directly from Lemma \ref{le sec3 approximate velocity}.

For $\overline{\theta}_2$, since $t \le \lambda^{\beta-1-\frac2p}\e^{-1-\frac{2}{\beta}}$, differentiating $\overline{\theta}_2$ introduces a multiplicative factor of at most $\lambda N \e^{-\frac{2}{\beta}}$, which implies that 
\begin{equation} 
	\begin{aligned}
		\|\overline{\theta}_2\|_{\dot{W}^{k,q}}\lesssim_{\e, k, q}  \lambda^{\frac2p-\frac2q-\beta}N^{-\beta}(\lambda N)^k
	\end{aligned}
\end{equation} 
for any $k\in\N$ and $1<q \le +\infty$. Hence the proof is complete.
\end{proof}

Now we can prove that the approximate solution $\overline{\theta}$ exhibits small-data norm inflation in the supercritical space $W^{\beta,p}$. Moreover, when $\beta=1$, it exhibits large-data norm inflation in the critical space $H^1$. 

For the convenience of the reader, although the argument is essentially elementary, we still provide the full details below.
\begin{proposition}[Norm inflation for $\overline{\theta}$] \label{prop norm inflation for approximate solution}
    For $1<p<2$ and $1 \le \beta <\frac2p$, there hold
    \begin{equation}\label{eq sec3 approximate inflation1}
    \|\overline{\theta}(0)\|_{W^{\beta,p}} \lesssim \e
    \end{equation}
    and 
    \begin{equation}\label{eq sec3 approximate inflation2}
    \|\overline{\theta}(T)\|_{\dot{W}^{\beta,p}} \gtrsim \e^{-1}.
    \end{equation}
    In addition, when $\beta=1$ and $1<p<2$, one has
    \begin{equation}\label{eq sec3 approximate inflation3}
    \|\overline{\theta}(0)\|_{H^1} \lesssim \e \lambda^{\frac2p-1}
    \end{equation}
    and 
    \begin{equation}\label{eq sec3 approximate inflation4}
    \|\overline{\theta}(T)\|_{\dot{H}^1} \gtrsim \e^{-1}\lambda^{\frac2p-1}.
    \end{equation}
\end{proposition}
\begin{proof}
    Inequalities \eqref{eq sec3 approximate inflation1} and \eqref{eq sec3 approximate inflation3} can be verified directly: the former follows by direct computation, while the latter relies on interpolation. 
    
    To estimate $\overline{\theta}(T)$,  we note that $\overline{\theta}_1$ does not depend on time $T$ and satisfies
    \begin{equation}\label{eq sec3 estimate theta1 Wbp}
        \|\overline{\theta}_1\|_{\dot{W}^{\beta,p}}\lesssim \e
    \end{equation}
    and
    \begin{equation}\label{eq sec3 estimate theta1 H1}
        \|\overline{\theta}_1\|_{\dot{H}^1} \lesssim \e \lambda^{\frac2p-1}.
    \end{equation}
    Thus, it suffices to estimate $\overline{\theta}_2(T)$. Recall that
    $$
    \overline{\theta}_2(T)=\e \lambda^{\frac2p-\beta}N^{-\beta} g(  \lambda r)\cos\left( N \alpha- N\e^{-\frac{2}{\beta}}h(\lambda r)  \right).
    $$
    A direct computation yields
    \begin{equation}\label{eq sec3 def nabla theta2}
        \nabla \overline{\theta}_2=\left(I_1+I_2\right)\vec{e}_{r}+I_3\vec{e}_{\alpha},
    \end{equation}
    where
    \begin{equation*}
        I_1=\e \lambda^{1+\frac2p-\beta}N^{-\beta} g'(   \lambda r)\cos\left( N \alpha- N\e^{-\frac{2}{\beta}}h(\lambda r)  \right),
    \end{equation*}
    \begin{equation*}
        I_2=\e^{1-\frac{2}{\beta}} \lambda^{1+\frac2p-\beta}N^{1-\beta} g( \lambda r)h'(\lambda r)\sin\left( N \alpha- N\e^{-\frac{2}{\beta}}h(\lambda r)  \right)
    \end{equation*}
    and
    \begin{equation*}
        I_3=-\e \lambda^{1+\frac2p-\beta}N^{1-\beta} \frac{g( \lambda r)}{ \lambda r}\sin\left( N \alpha- N\e^{-\frac{2}{\beta}}h(\lambda r)  \right).
    \end{equation*}
    Thus, for any $1<q \le \infty$, a straightforward estimate gives
    \begin{equation}\label{eq sec3 estimate I1 and I3}
        \|I_1\|_{L^q}+\|I_3\|_{L^q} \lesssim \e\lambda^{1+\frac2p-\frac2q-\beta}
    \end{equation}
    and (using the properties of $f, g$ and $h$ from Section~\ref{se sec3 definition 3.1})
    \begin{equation}\label{eq sec3 estimate I2}
        \|I_2\|_{L^q} \approx \e^{1-\frac{2}{\beta}}\lambda^{1+\frac2p-\frac2q-\beta}N^{1-\beta}.
    \end{equation}
    Setting $\beta=1$ and $q=p$ in the above expressions, we obtain
    $$
    \|\overline{\theta}_2\|_{\dot{H}^1} \approx \e^{-1}\lambda^{\frac2p-1}.
    $$
    Together with \eqref{eq sec3 estimate theta1 H1}, we obtain \eqref{eq sec3 approximate inflation4}.
    
    For \eqref{eq sec3 approximate inflation2}, the case $\beta=1$ follows from the same calculation. When $1<\beta<\frac2p$, we first note that
    $$
    \| \overline{\theta}_2(T)\|_{L^p} \lesssim \e \lambda^{-\beta}N^{-\beta},
    $$
    Using \eqref{eq sec3 estimate I1 and I3}, \eqref{eq sec3 estimate I2}, together with the interpolation inequality
    $$
    \|\overline{\theta}_2(T)\|_{\dot{W}^{1,p}} \lesssim \|\overline{\theta}_2(T)\|_{L^p}^{1-\frac{1}{\beta}}\|\overline{\theta}_2(T)\|_{\dot{W}^{\beta,p}}^{\frac{1}{\beta}},
    $$
    we have
    $$
    \|\overline{\theta}_2(T)\|_{\dot{W}^{\beta,p}} \gtrsim \e^{-1}.
    $$
    Combined with \eqref{eq sec3 estimate theta1 Wbp}, we obtain \eqref{eq sec3 approximate inflation2} and hence completes the proof.
\end{proof}

\subsection{Error estimates}

We now estimate the the approximation error   $\overline{F}$. Recall that  $\overline{F}$ is defined by  
\begin{equation}\begin{aligned}\label{eq sec3 error def}
	\overline{F}&=u[\overline{\theta}_2]\cdot \nabla \overline{\theta}_1+u[{\overline{\theta}_2}]\cdot\nabla \overline{\theta}_2+\Lambda \overline{\theta}_1+\Lambda \overline{\theta}_2 \\
	&=: \overline{F}_1+\overline{F}_2+\overline{F}_3+\overline{F}_4.
\end{aligned}\end{equation}

\begin{proposition}\label{prop sec3 error estimate}
	For any $1< q <  \infty$ and $  k \ge 0 $, there holds
	\begin{equation}\label{eq sec3 error es}
		\|\overline{F}(t)\|_{\dot{W}^{k,q}} \lesssim_{\e, k, q} (\lambda N )^k \lambda^{1+\frac4p-\frac2q-2\beta}N^{-1-\beta}
	\end{equation}
provided that $N$ is chosen large enough.
\end{proposition}
\begin{proof}
Again, it suffices to consider $k \in \N$ , then \eqref{eq sec3 error es} follows by interpolation. For brevity, below we omit the subscript $\e, q, k$.

By Proposition \ref{prop est bar theta}, when $1 <q <  \infty$ , 
\begin{equation}\label{eq es for F4}
	\|\overline{F}_4\|_{\dot{W}^{k,q}} \lesssim \|\overline{\theta}_2\|_{\dot{W}^{k+1,q}} \lesssim \lambda^{\frac2p-\frac{2}{q}-\beta}N^{-\beta}(\lambda N)^{k+1}
\end{equation}
and
	\begin{equation}\label{eq es for F3}
	\|\overline{F}_3\|_{\dot{W}^{k,q}} \lesssim \|\overline{\theta}_1\|_{\dot{W}^{k+1,q}} \lesssim \lambda^{k+1} \lambda^{\frac2p-\frac2q-\beta}.
	\end{equation}
Recalling from \eqref{eq:parameters lambda} that $\lambda^{\beta-\frac2p}= N^{-100}$, we conclude that $\overline{F}_3$ and $\overline{F}_4$ satisfy the desired bound.

Next, we estimate the quadratic errors   $F_1$ and $F_2$, using the approximate operators defined in \eqref{eq:appro_u} and \eqref{eq sec3 psi2}.

\noindent
\textbf{Estimate for $\overline{F}_1$}

Since $\operatorname{supp}(\overline{\theta}_1) \cap \operatorname{supp}(\overline{\theta}_2) = \emptyset$, we can write $\overline{F}_1$ as
$$
\overline{F}_1=(u[\overline{\theta}_2]-\overline{u}[\overline{\theta}_2])\cdot \nabla \overline{\theta}_1 
$$
where $\overline{u}[\overline{\theta}_2])$ is the approximate velocity of $\overline{\theta}_2$ defined in \eqref{eq:appro_u} and \eqref{eq sec3 psi2}.

Therefore, by Theorem \ref{thm:approx velocity}, we obtain
\begin{equation}\label{eq es for F1}
	\begin{aligned}
		\|\overline{F}_1\|_{\dot{W}^{k,p}}&\lesssim \sum_{i=0}^k \| u[\overline{\theta}_2]-\overline{u}[\overline{\theta}_2] \|_{\dot{W}^{i,p}} \| \overline{\theta}_1 \|_{C^{k+1-i}} \\ 
		&\lesssim (\lambda N )^k \lambda^{1+\frac4p-\frac2q-2\beta}N^{-1-\beta}.
	\end{aligned}
\end{equation}

\noindent
\textbf{Estimate for $\overline{F}_2$}

We use \eqref{eq:appro_u} and \eqref{eq sec3 psi2} to decompose $\overline{F}_2$ as
$$
\overline{F}_2=(u[\overline{\theta}_2]-\overline{u}[\overline{\theta}_2])\cdot \nabla \overline{\theta}_2+\overline{u}[\overline{\theta}_2]\cdot \nabla \overline{\theta}_2=:\overline{F}_{2,1}+\overline{F}_{2,2}.
$$
For $\overline{F}_{2,1}$, using an argument similar to that for$\overline{F}_1$, we get
\begin{equation}\label{eq es for F21}
	\|\overline{F}_{2,1}\|_{\dot{W}^{k,p}}\lesssim (\lambda N )^k \lambda^{1+\frac4p-\frac2q-2\beta}N^{-2\beta}.
\end{equation}
For $\overline{F}_{2,2}$, from \eqref{eq:appro_u} and \eqref{eq sec3 psi2}, we can write
$$
 \overline{u}[\overline{\theta}_2] = \nabla^{\perp}\left( \frac{\overline{\theta}_2}{N\lambda\psi_{2,1}(t,\lambda r)} \right),
$$
which implies
\begin{equation}\label{eq es for F22 3}
	\begin{aligned}
		\overline{F}_{2,2} = \overline{u}[\overline{\theta}_2]\cdot \nabla \overline{\theta}_2 =\overline{\theta}_2 \nabla^{\perp}\left( \frac{1 }{N \lambda \psi_{2,1}(t, \lambda r)}    \right) \cdot \nabla \overline{\theta}_2.
	\end{aligned}
\end{equation}
Using \eqref{eq es for F22 2}, we have
\begin{equation}\label{eq es for F22 4}
\left\|	\left( \frac{1 }{N \lambda \psi_2(t, \lambda r)}    \right) \right\|_{C^k(\operatorname{supp} \overline{\theta}_2)} \lesssim N^{-1}\lambda^{k-1}.
\end{equation}
Combining \eqref{eq es for F22 3}, \eqref{eq es for F22 4}, and Proposition \ref{prop est bar theta}, we obtain the desired bound for $\overline{F}_{2,2}$:
\begin{equation}
	\|\overline{F}_{2,2}\|_{\dot{W}^{k,q}}\lesssim (\lambda N )^k \lambda^{1+\frac4p-\frac2q-2\beta}N^{-2\beta}.
\end{equation}
\end{proof}

%%%%%%%%%%%%%%%%%%%%%%%%%%%%%%%%%%%%%%%%%%%%%%%%%%%%%%%%%%%%%%%%%%%%%%%%%%%%%%%%%%%%
\section{Proof of Theorem \ref{thm:critical}: stability and conclusion}
%%%%%%%%%%%%%%%%%%%%%%%%%%%%%%%%%%%%%%%%%%%%%%%%%%%%%%%%%%%%%%%%%%%%%%%%%%%%%%%%%%%%
We are now in the position to show that the exact solution $\theta$ with initial data $\overline{\theta}(0) $ satisfies the conclusions of Theorem \ref{thm:critical} and Theorem \ref{thm:supercritical}.

\subsection{Energy estimates}

Define the difference $\eta : = \theta-\overline{\theta} $. By the global regularity of $\theta$ and $\overline{\theta}$, it follows that $\eta $ satisfies
\begin{equation}
\begin{cases}
	\partial_t \eta+\Lambda \eta +u[\overline{\theta}]\cdot \nabla \eta+u[\eta]\cdot \nabla \eta =-\overline{F}-u[\eta]\cdot \nabla \overline{\theta} &\\
\eta |_{t = 0} = 0.
\end{cases}
\end{equation}

\begin{theorem}\label{thm estimate for error}
	For any $0 \le k \le 2$, $1 < q < \infty$ and $t \le T$, there holds
	\begin{equation}
		\|\eta(t)\|_{W^{k,q}} \lesssim_{\e, k, q} (\lambda N)^k \lambda^{\frac2p-\frac2q-\beta}N^{ -1-\beta}.
	\end{equation}
\end{theorem}
\begin{proof}
Again it suffices to consider $k \in \N$. To show the smallness of $\eta$, let us define
\begin{equation}\label{eq tcrit def}
	T_{*}=\sup \left\{ t\le T:  \| \nabla u[\eta]\|_{L^{\infty}}+\| \nabla \eta\|_{L^{\infty}} \le \lambda^{1+\frac2p-\beta} \right\}.
\end{equation}
By continuity, $T_{*}>0 $. We first derive estimates on $[0,T_*]$, then use a continuity argument to show $T_*=T$.

	\emph{Step 1: $L^q$ estimates.} Thanks to the C\'{o}rdoba-C\'{o}rdoba inequality  \cite[Lemma $2.5$]{CC}, direct computation yields
	\begin{equation}
		\frac{d \|\eta\|_{L^q}}{dt} \lesssim \|\overline{F}\|_{L^q}+\|\nabla \overline{\theta}\|_{L^{\infty}}\|\eta\|_{L^q}.
	\end{equation}
	Applying Gronwall's inequality, together with Proposition \ref{prop est bar theta} and \ref{prop sec3 error estimate}, we obtain for all $t\le T_*$, (throughout the proof, the constant $C_{\e}$ will change from line-to-line but is independent of $\lambda, N$ and $t$)
	\begin{equation}\label{eq l2 estimate for error}
		\begin{aligned}
			\|\eta(t)\|_{L^q} \lesssim& \, t\|\overline{F}\|_{L^q}e^{C_{\e}\|\nabla \overline{\theta}\|_{L^{\infty}}t} \\
			\lesssim& \lambda^{\frac2p-\frac2q-\beta} N^{-1-\beta}e^{C_{\e}}.
		\end{aligned}
	\end{equation}
	
	\emph{Step 2: $W^{1,q}$ estimates.} Differentiating $\eta$ with respect to $x_i$, we get
	\begin{equation}
		\begin{aligned}
		\partial_t \partial_i\eta=&-\partial_iu[\overline{\theta}]\cdot \nabla \eta-u[\overline{\theta}]\cdot \nabla \partial_i\eta-\partial_iu[\eta]\cdot \nabla \eta\\&-u[\eta]\cdot \nabla \partial_i\eta -\Lambda \partial_i\eta-\partial_i\overline{F}-\partial_iu[\eta]\cdot \nabla \overline{\theta}-u[\eta]\cdot \nabla \partial_i\overline{\theta},
		\end{aligned}
	\end{equation}
which yields 
\begin{equation}\begin{aligned}
\frac{d\| \nabla \eta\|_{L^q}}{dt} \lesssim& \left( \|\nabla \overline{\theta}\|_{L^{\infty}}+\|\nabla \eta\|_{L^{\infty}}+\|\nabla u[\overline{\theta}]\|_{L^{\infty}} \right) \|\nabla \eta\|_{L^q}\\&+\|\nabla \overline{F}\|_{L^q }+\| \nabla^2\overline{\theta}\|_{L^{\infty}}\|\eta\|_{L^q}.
\end{aligned}\end{equation}
By Gronwall's inequality, \eqref{eq tcrit def}, \eqref{eq l2 estimate for error} and Proposition \ref{prop est bar theta} and \ref{prop sec3 error estimate}, we find 
\begin{equation}\label{eq H1 estimate for error}
	\begin{aligned}
		\|\nabla \eta\|_{L^q} \lesssim (\lambda N) \lambda^{\frac2p-\frac2q-\beta}N^{-1-\beta}e^{C_{\e}} \quad \text{ for all $t\le T_*$}.
	\end{aligned}
\end{equation}

\emph{Step 3: $W^{2,q}$ estimates.} Via a similar argument, we obtain
\begin{equation}
	\begin{aligned}
		\frac{d \| \nabla^2 \eta \|_{L^q}}{dt} \lesssim& \left( \| \nabla \overline{\theta}\|_{L^{\infty}}+\|\nabla u[\overline{\theta}] \|_{L^{\infty}}+\| \nabla \eta\|_{L^{\infty}}+\|\nabla u[\eta] \|_{L^{\infty}} \right) \| \nabla^2 \eta\|_{L^q} \\
		&+\left( \|\nabla^2 u[\overline{\theta}]\|_{L^{\infty}}+\|\nabla^2\overline{\theta}\|_{L^{\infty}} \right)\|\nabla \eta\|_{L^q}+\|\nabla^3\overline{\theta}\|_{L^{\infty}}\|\eta\|_{L^q} \\
		&+\|\nabla^2\overline{F}\|_{L^q},
	\end{aligned}
\end{equation}
which gives (utilizing \eqref{eq tcrit def}, \eqref{eq l2 estimate for error}, \eqref{eq H1 estimate for error} and Proposition \ref{prop est bar theta} and \ref{prop sec3 error estimate})
\begin{equation}\label{eq H2 estimate for error}
	\|\nabla^2 \eta \|_{L^q} \lesssim (\lambda N)^2 \lambda^{\frac2p-\frac2q-\beta}N^{-1-\beta}e^{C_{\e}} \quad \text{ for all $t\le T_*$}.
\end{equation}
Therefore, it remains to show that $T_*=T$.

By Lemma \ref{le interpolation} (setting $-\delta_1=\delta_2=\frac2p=\frac12$ and $\theta=\frac12$ ), there holds
\begin{equation}
	\| \nabla u[\eta]\|_{L^{\infty}}+\| \nabla \eta\|_{L^{\infty}} \lesssim \|\eta\|_{W^{2,4}}^{\frac12}\|\eta\|_{W^{1,4}}^{\frac12}.
\end{equation} 
Then for any $t \le T_{*}$, applying \eqref{eq l2 estimate for error}, \eqref{eq H1 estimate for error} and \eqref{eq H2 estimate for error} gives
\begin{equation}
	\begin{aligned}
		\| \nabla u[\eta](t)\|_{L^{\infty}}+\| \nabla \eta(t)\|_{L^{\infty}} \le C_{\e} N^{\frac12-\beta}\lambda^{1+\frac2p-\beta}.
	\end{aligned}
\end{equation}
Thus, if we choose $N$ sufficiently large so that both 
$N^{-\frac12} C_{\varepsilon} \le \frac12$ and the conditions of Proposition 3.4 are satisfied, then it follows from the continuity of the solution and the maximality of $T_*$ that $T_*=T$ and hence the proof is complete.
\end{proof}
\subsection{Conclusion of the proof}
 
With all the ingredients in hand, we can conclude the proofs of the main two theorems. We first establish small-data norm inflation in $W^{\beta,p}$.

\begin{proof}[Proof of Theorem \ref{thm:supercritical}]
    Recall that $\theta_0=\overline{\theta}(0)$, using Proposition \ref{prop norm inflation for approximate solution}, we see that
    $$
    \|\theta_0\|_{W^{\beta,p}}=\|\overline{\theta}(0)\|_{W^{\beta,p}} \lesssim \e.
    $$
    Utilizing Proposition \ref{prop norm inflation for approximate solution} and Theorem \ref{thm estimate for error}, by taking $N$  large enough, there hold
    $$
    \|\theta(T)\|_{W^{\beta,p}} \ge \|\overline{\theta}(T)\|_{W^{\beta,p}}-\|\eta(T)\|_{W^{\beta,p}} \gtrsim \e^{-1}
    $$
     and the proof is complete.
\end{proof}
Next, we prove large-data norm inflation in $H^1$.

\begin{proof}[Proof of Theorem \ref{thm:critical}]
    We first set $\beta=1$ and take $1<p<2$ close to $2$, so that the first claim \eqref{eq thm:critical 1} follows from the embedding $W^{\beta,p} \hookrightarrow H^{1-\e} $.
    
    As for the second claim, Proposition \ref{prop norm inflation for approximate solution} gives
    $$
    \|\theta_0\|_{H^1} \lesssim \e \lambda^{\frac2p-1}.
    $$
    Again, utilizing Proposition \ref{prop norm inflation for approximate solution} and Theorem \ref{thm estimate for error}, we have
    $$
    \|\theta(T)\|_{H^1} \ge \|\overline{\theta}(T)\|_{H^1}-\|\eta(T)\|_{H^1} \gtrsim \e^{-1} \lambda^{\frac2p-1},
    $$
    which proves
    $$
    \frac{\|\theta(T)\|_{H^1}}{\|\theta_0\|_{H^1}} \gtrsim \e^{-2}.
    $$

\end{proof}

%%%%%%%%%%%%%%%%%%%%%%%%%%%%%%%%%%%%%%%%%%%%%%%%%%%%%%%%%%%%%%%%%%%%%%%%%%%%%%%%%%%%
\section{Extensions to the periodic domain $\T^2$}
%%%%%%%%%%%%%%%%%%%%%%%%%%%%%%%%%%%%%%%%%%%%%%%%%%%%%%%%%%%%%%%%%%%%%%%%%%%%%%%%%%%%
For the critical SQG equation on $\mathbb{T}^2 \times [0,+\infty)$, the strategy is much similar to that on $\mathbb{R}^2$. We will just periodize the $\mathbb{R}^2$ approximate solution and show that 
\begin{itemize}
    \item The error term satisfies the same estimate as that in Proposition \ref{prop sec3 error estimate}.
    \item The approximate solution exhibits norm inflation.
\end{itemize}

\subsection{Periodization}

To clearly distinguish the original $\mathbb{R}^2$ approximate solution from its periodic extension (to $\mathbb{T}^2$) in the current section, we denote the $\mathbb{R}^2$ approximate solution as $\overline{\theta}_c$ (as in \eqref{eq:def_theta bar}--\eqref{eq:def_theta bar 2}). To avoid ambiguity of the linear term $\Lambda$ on $\T^2$, we slightly rearrange the approximate equations:
\begin{equation}
 \partial_t \overline{\theta}_c + u[\overline{\theta}_c] \cdot \nabla \overline{\theta}_c = \overline{F}_c 
\end{equation} 
where  $\overline{F}_c$ coincides with $\overline{F}_1+\overline{F}_2$ in \eqref{eq sec3 error def}, which has already been estimated in Proposition \ref{prop sec3 error estimate}, so we have
\begin{proposition}\label{prop sec5 error estimate}
For all $N $ sufficiently large, the supports of $\Bar{\theta}_c$ and $\overline{F}_c$ are contained in $B_{\frac{1}{100}}(0)$. For any $1<q< \infty$ and $k \ge 0$, there holds
   \begin{equation}\label{eq sec5 error es}
		\|\overline{F}_c\|_{\dot{W}^{k,q}(\R^2)} \lesssim_{\e, k, q} (\lambda N )^k \lambda^{1+\frac4p-\frac2q-2\beta}N^{-1-\beta}.
	\end{equation}
\end{proposition}

Then we extend $\overline{\theta}_c$ to a $2\pi$-periodic function on $\mathbb{T}^2=\mathbb{R}^2/2\pi\mathbb{Z}^2$. We define $\overline{\theta}: \T^2 \times [ 0, \infty) \to \R$    by
\begin{equation}\label{eq sec5 approx theta det}
\overline{\theta}(x,t)=  \sum_{n \in \mathbb{Z}^2}
\overline{\theta}_c(x+ 2 \pi n,t).
\end{equation}

The first goal is to show the $2\pi$-periodic function $\overline{\theta} $ is a periodic approximate solution (as in Proposition \ref{prop norm inflation for approximate solution}).
\begin{proposition}[Norm inflation for $\overline{\theta}$] \label{prop T2 norm inflation for approximate solution}
    For $1<p<2$ and $1 \le \beta <\frac2p$, there hold
    \begin{equation}\label{eq sec5 T2 approximate inflation1}
    \|\overline{\theta}(0)\|_{W^{\beta,p}(\T^2)} \lesssim \e
    \end{equation}
    and 
    \begin{equation}\label{eq sec5 T2 approximate inflation2}
    \|\overline{\theta}(T)\|_{\dot{W}^{\beta,p}(\T^2)} \gtrsim \e^{-1}.
    \end{equation}
    Moreover, setting $\beta=1$ and $1<p<2$, we have
    \begin{equation}\label{eq sec5 T2 approximate inflation3}
    \|\overline{\theta}(0)\|_{H^1(\T^2)} \lesssim \e \lambda^{\frac2p-1}
    \end{equation}
    and 
    \begin{equation}\label{eq sec5 T2 approximate inflation4}
    \|\overline{\theta}(T)\|_{\dot{H}^1(\T^2)} \gtrsim \e^{-1}\lambda^{\frac2p-1}.
    \end{equation}
\end{proposition}
\begin{proof}
    Note that 
    $$
    \|\overline{\theta}\|_{W^{s,q}(\T^2)}  = \|\overline{\theta}_c\|_{W^{s,q}(\R^2)}
    $$
    when $s \in \N$. So \eqref{eq sec5 T2 approximate inflation3} and \eqref{eq sec5 T2 approximate inflation4} follow directly from Proposition \ref{prop norm inflation for approximate solution}. 
    
    Using interpolation, we obtain
    $$
    \|\overline{\theta}(0)\|_{W^{\beta,p}(\T^2)} \lesssim \|\overline{\theta}(0)\|_{L^p(\T^2)}^{1-\beta}\|\overline{\theta}(0)\|_{W^{1,p}(\T^2)}^{\beta}\approx \|\overline{\theta}_c(0)\|_{L^p(\T^2)}^{1-\beta}\|\overline{\theta}_c(0)\|_{W^{1,p}(\T^2)}^{\beta} \lesssim \e,
    $$
    which proves \eqref{eq sec5 T2 approximate inflation1}. 
    
    Finally, \eqref{eq sec5 T2 approximate inflation2} can be verified directly when $\beta=1$. For $\beta \neq 1$, \eqref{eq sec5 T2 approximate inflation2} is a consequence of the interpolation
    $$
    \|\overline{\theta}(T)\|_{\dot{W}^{1,p}(\T^2)} \lesssim \|\overline{\theta}(T)\|_{L^p(\T^2)}^{1-\frac{1}{\beta}}\|\overline{\theta}(T)\|_{\dot{W}^{\beta,p}(\T^2)}^{\frac{1}{\beta}}.
    $$
\end{proof}

To derive estimates on the approximation error of $\overline{\theta}$, let us first analyze the periodic velocity. According to \eqref{eq sec5 approx theta det}, the approximate velocity $\overline{u}:=u[\overline{\theta}]$ admits the decomposition
$$\overline{u}(x,t)=\overline{u}_0(x,t)+\sum_{n\in \mathbb{Z}^2\setminus (0,0)}\overline{u}_n(x,t),$$
where
$$
\overline{u}_0(x,t):=u[\overline{\theta}_c](x,t)=\overline{u}_c(x,t)
$$
and for $n\neq (0,0)$,
$$
\overline{u}_n(x,t):=c \int_{\R^2} \frac{(x-y)^{\perp}}{|x-y|^3} \overline{\theta}_n(y,t)\,dy.
$$
\begin{lemma}\label{le sec5 approximate velocity}
    For any $k\in \N$, there holds
    $$
    \|\nabla^k \sum_{n\in \mathbb{Z}^2\setminus (0,0)}\overline{u}_n(x,t)\|_{L^{\infty}} \lesssim_{\e,k} \lambda^{\frac2p-2-\beta}.
    $$
\end{lemma}
\begin{proof}
    For any $n \neq (0,0)$, we have
\begin{equation}
    \begin{aligned}
        \overline{u}_n(x)=&c \int_{\R^2} \frac{(x_1-y_1-2n_1\pi,x_2-y_2-2n_2\pi)^{\perp}}{|x-y-(2n_1\pi,2n_2\pi)|^3} \overline{\theta}_c(y,t)\,dy\\
        =&c \int_{\R^2} \frac{(x_1-y_1,x_2-y_2)^{\perp}}{|x-y-(2n_1\pi,2n_2\pi)|^3} \overline{\theta}_c(y,t)\,dy\\&+c \int_{\R^2} \frac{(-2n_1\pi,-2n_2\pi)^{\perp}}{|x-y-(2n_1\pi,2n_2\pi)|^3} \overline{\theta}_c(y,t)\,dy\\
        =&c \int_{\R^2} \frac{(-2n_1\pi,-2n_2\pi)^{\perp}}{|x-y-(2n_1\pi,2n_2\pi)|^3} \overline{\theta}_c(y,t)\,dy+O(\frac{\|\overline{\theta}_c\|_{L^1}}{|n|^3}). 
    \end{aligned}
\end{equation}
Summing over $n \in \mathbb{Z}^2\setminus (0,0)$ and using the symmetry $n\to -n$, we get 
\begin{equation}
    \begin{aligned}
        \sum_{n\in \mathbb{Z}^2\setminus (0,0)}\overline{u}_n(x,t)=&\sum_{n\in \mathbb{Z}^2\setminus (0,0)}\frac{\overline{u}_n(x,t)+\overline{u}_{-n}(x,t)}{2} \\
        =&\sum_{n\in \mathbb{Z}^2\setminus (0,0)}\frac{\Gamma(\frac32)}{2\pi^{\frac32}}\int_{\R^2} (-2n_1\pi,-2n_2\pi)^{\perp}I_n(x,y) \overline{\theta}_c(y,t)\,dy\\
        &+O(\frac{\|\overline{\theta}_c\|_{L^1}}{|n|^3}), 
    \end{aligned}
\end{equation}
where (using $x \in \T^2$ and $y \in \operatorname{Supp} \overline{\theta}_c$)
$$
|I_n(x,y)|= \left|\frac{1}{|x-y-(2n_1\pi,2n_2\pi)|^3}-\frac{1}{|x-y+(2n_1\pi,2n_2\pi)|^3}\right| \lesssim \frac{1}{|n|^4}.
$$
Thus,
\begin{equation}
    \begin{aligned}
        \sum_{n\in \mathbb{Z}^2\setminus (0,0)}\overline{u}_n(x,t)=&O(\frac{\|\overline{\theta}_c\|_{L^1}}{|n|^3}) =\sum_{n\in \mathbb{Z}^2\setminus (0,0)}O(\frac{\lambda^{\frac2p-2-\beta}}{|n|^3}) =O(\lambda^{\frac2p-2-\beta}).
    \end{aligned}
\end{equation}
For $k \ge1$, a similar but simple argument gives the same bound:
\begin{equation}
    \begin{aligned}
        \nabla^k\sum_{n\in \mathbb{Z}^2\setminus (0,0)}\overline{u}_n(x,t)=&\sum_{n\in \mathbb{Z}^2\setminus (0,0)}\nabla^k\overline{u}_n(x,t)\\
        =&O(\frac{\|\overline{\theta_c}\|_{L^1}}{|n|^3}) \\
        =&O(\lambda^{\frac2p-2-\beta}).
    \end{aligned}
\end{equation}
\end{proof}

With the preparations in hand, we can show that the periodized solution $\overline{\theta} $ gives rise to a small error similar to that of Proposition \ref{prop sec3 error estimate}.
Define the error
 $$
 \overline{F}= \partial_t \overline{\theta}+\overline{u}\cdot \nabla \overline{\theta}+\Lambda \overline{\theta},
$$
then we have
\begin{proposition}
    For any $1< q <  \infty$ and $  k \ge 0 $, there holds
	\begin{equation}\label{eq sec6 error es T^2}
		\|\overline{F}\|_{\dot{W}^{k,q}(\T^2)} \lesssim_{\e, k, q} (\lambda N )^k \lambda^{1+\frac4p-\frac2q-2\beta}N^{-1-\beta}
	\end{equation}
provided that $N$ is chosen large enough.
\end{proposition}

\begin{proof} We only consider $k \in \N$ as the general case can be obtained easily via interpolation. For any $x \in \mathbb{T}^2$, a direct calculation gives
\begin{equation}
    \begin{aligned}
        \overline{F}=& \partial_t \overline{\theta}+\overline{u}\cdot \nabla \overline{\theta}+\Lambda \overline{\theta} \\
        =& \partial_t \overline{\theta}_c+\overline{u}_c\cdot \nabla \overline{\theta}_c+\Lambda \overline{\theta}+\sum_{n\in \mathbb{Z}^2\setminus (0,0)}\overline{u}_n \cdot \nabla \overline{\theta}_c\\
        =& \overline{F}_c+\Lambda \overline{\theta}+\sum_{n\in \mathbb{Z}^2\setminus (0,0)}\overline{u}_n \cdot \nabla \overline{\theta}_c.
    \end{aligned}
\end{equation}
We see from Proposition \ref{prop sec5 error estimate} that $\overline{F}_c$ satisfies the desired estimate. For the dissipation term, using Proposition \ref{prop est bar theta} and recalling that $\lambda^{\frac2p-\beta}=N^{100}$, we obtain
    \begin{align*}
        \|\Lambda \overline{\theta}\|_{W^{k,q}(\mathbb{T}^2)} \lesssim& \|\nabla \overline{\theta}\|_{W^{k,q}(\mathbb{T}^2)}= \|\nabla \overline{\theta}_c\|_{W^{k,q}(\mathbb{R}^2)}  \ll (\lambda N )^k \lambda^{1+\frac4p-\frac2q-2\beta}N^{-1-\beta}.
    \end{align*}
    
 Now it remains to estimate $\sum_{n\in \mathbb{Z}^2\setminus (0,0)}\overline{u}_n(x,t) \cdot \nabla \overline{\theta}_c$, which follows from Proposition \ref{prop est bar theta} and Lemma \ref{le sec5 approximate velocity} that
\begin{align*}
    \|\sum_{n\in \mathbb{Z}^2\setminus (0,0)}\overline{u}_n \cdot \nabla \overline{\theta}_c\|_{W^{k,q}(\mathbb{T}^2)} \lesssim& \lambda^{\frac2p-2-\beta}\|\nabla \overline{\theta}_c\|_{W^{k,q}(\R^2)} \ll (\lambda N )^k \lambda^{1+\frac4p-\frac2q-2\beta}N^{-1-\beta}.
\end{align*}
Combining these estimates completes the proof.
\end{proof}
Then arguing as in Theorem \ref{thm estimate for error}, we can obtain the estimate for the remainder term. Together with Proposition \ref{prop T2 norm inflation for approximate solution}, we prove Theorem \ref{thm:critical} and \ref{thm:supercritical} on the periodic domain $\T^2$.

\appendix

%%%%%%%%%%%%%%%%%%%%%%%%%%%%%%%%%%%%%%%%%%%%%%%%%%%%%%%%%%%%%%%%%%%%%%%%%%%%%%%%%%%%
\section{Useful lemmas}
%%%%%%%%%%%%%%%%%%%%%%%%%%%%%%%%%%%%%%%%%%%%%%%%%%%%%%%%%%%%%%%%%%%%%%%%%%%%%%%%%%%%
\begin{lemma}\label{le fgh}
	There exists a smooth radial function $f : \R^2 \to \R$ such that  $\operatorname{supp} f \subset (\frac12,2)$ and
	$$\partial_rh(x_0):=\partial_r \left(\frac{u[f](x_0)}{r} \cdot \vec{e}_{\alpha}\right) \neq 0$$
	for some $x_0$ with $|x_0|\ge 4$.
\end{lemma}
\begin{proof}
    Indeed, for any non-negative function $f$ supported in $[\frac{1}{2},2]$, we have
    $$
    \frac{u[f](x)}{r} \cdot \vec{e}_{\alpha}=\frac{x^{\perp}}{|x|^2}\int_{\R^2} \frac{(x-y)^{\perp}}{|x-y|^3}f(y)\,dy,
    $$
    which behaves asymptotically like
    $$
    \frac{\int_{\R^2}f(y)\,dy}{r^3}
    $$
    as $r \to \infty$. Thus, there must exist some $|x_0|\ge 4$ such that
    $$\partial_rh(x_0):=\partial_r \left(\frac{u[f](x_0)}{r} \cdot \vec{e}_{\alpha}\right) \neq 0.$$
\end{proof}

\subsection{Proof of Lemma \ref{le interpolation}}
 
    Recalling that
    \[
\dot{B}^s_{p,1} \hookrightarrow \dot{W}^{s,p} \hookrightarrow \dot{B}^s_{p,\infty},
\]
	we obtain
		$$
		\|f\|_{\dot{W}^{s,\infty}}\lesssim \|f\|_{\dot{B}^s_{\infty,1}}.
		$$	
		Therefore, setting $\theta \delta_1+(1-\theta)\delta_2=0$ gives
		\begin{equation}
			\begin{aligned}
				\|f\|_{\dot{W}^{s,\infty}} \lesssim& \|f\|_{\dot{B}^{\theta s_1 +(1-\theta)s_2}_{\infty,1}},
			\end{aligned}
		\end{equation}
		where $s_1=s+\delta_1$ and $s_2=s+\delta_2$. Then it follows from \cite[Proposition $2.22$ and Theorem $2.41$]{BCD} that
		\begin{equation}
			\begin{aligned}
				\|f\|_{\dot{W}^{s,\infty}} \lesssim& \|f\|_{\dot{B}^{\theta s_1 +(1-\theta)s_2}_{\infty,1}}\\ \lesssim& \|f\|_{\dot{B}_{\infty,\infty}^{s_1}}^{\theta}\|f\|_{\dot{B}_{\infty,\infty}^{s_2}}^{1-\theta} \\
				\lesssim&\|f\|_{\dot{B}_{p,\infty}^{s_1+\frac{2}{p}}}^{\theta}\|f\|_{\dot{B}_{p,\infty}^{s_2+\frac{2}{p}}}^{1-\theta} \\
				\lesssim&  \|f\|_{\dot{W}^{s_1+\frac{2}{p},p}}^{\theta}\|f\|_{\dot{W}^{s_2+\frac{2}{p},p}}^{1-\theta},
			\end{aligned}
		\end{equation}

where $\theta=\frac{\delta_2}{\delta_2-\delta_1}$.

\subsection{Proof of Lemma \ref{le osc}}

\begin{proof}
It suffices to prove the estimates for $k=0$. Indeed, for $k\in \N$, the results follow by differentiating and repeatedly applying the $k=0$ case.

 Consider the complex-valued function $ g(x) e^{iN  \phi(x)}$ with $\phi (x): =\alpha + h(r)$, and rewrite the fractional operator as the oscillatory integral:
\begin{align*}
 \Lambda^{-s} \theta & : = \int_{\R^2_{\xi}} \int_{\R^2_{y}} g(y) e^{iN \phi(y)} e^{i(x-y)\cdot \xi} a_s(\xi) \, dy \, d\xi
 \end{align*}
 where $a_s(\xi) = |\xi|^{-s}$.

 Change of variable $\xi\mapsto \xi' N$ and define $\Psi (x; y,\xi) = \phi(y) +(x-y) \cdot \xi $,
 \begin{align*} 
 \Lambda^{-s} \theta&  = N^{2-s }\int_{\R^2_{\xi}} \int_{\R^2_{y}} g(y) e^{iN \phi(y)} e^{iN(x-y)\cdot \xi} a_s(\xi) \, dy \, d\xi \\
 &  = N^{2-s }\int_{\R^2_{\xi}} \int_{\R^2_{y}} g(y) a_s(\xi) e^{iN \Psi }    \, dy \, d\xi \\
 & : = N^{2-s } I_N(x)
\end{align*}
\paragraph{Estimates when $|x|\le 2\gamma$.}
We use stationary phase approximation to evaluate $I_N$. To separate the stationary point, let us introduce a partition of unity $\sum_{1\leq i \leq 3}\chi_i(\xi) $:
\begin{itemize}
    \item Let $\delta=\frac{\inf_{x\in \operatorname{supp} g} |\nabla \alpha|}{100}=\inf_{x\in \operatorname{supp} g}\frac{1}{100r}$.
    \item $ \chi_1(\xi) =1 $ if $ |\xi|   \leq \delta   $ and $ \chi_3(\xi) =1 $  if $|\xi| \geq \delta^{-1}$.
    \item $\chi_2$ is supported on $\frac{\delta}{2}\le |\xi| \le 2\delta^{-1}$.
\end{itemize}
With this partition, we decompose
\begin{equation}
    I_N(x) = I_{N,1} + I_{N,2} + I_{N, 3}.
\end{equation}

\noindent
\textbf{Estimate of $I_{N,3}$}

We will show that $| I_{N,3}| \lesssim N^{-2-s}$.

In this regime $ |\nabla_{y,\xi} \Psi | \geq |\xi - \nabla \phi | \geq \delta$.
So we consider the operator
\begin{equation}
    L_y : =  -i  \frac{\nabla\phi(y) - \xi}{ | \nabla\phi(y) - \xi|^2 } \cdot \nabla_y 
\end{equation}
satisfying
\begin{equation}
\frac{L_y}{N} ( e^{i N \Psi} )= e^{i N \Psi}.
\end{equation}

By direct computation, for each $n\in \mathbb{N} $, we have uniform in $y $ bounds:
\begin{equation}
    |[L_y^*]^n (g) | \lesssim_{\delta, g, h}  \langle\xi\rangle^{-2}.
\end{equation}
Integrating by parts $n = \lfloor{ s+2} \rfloor+1$ times, we have
\begin{align*}
|I_{N,3} | & = \Big| \int_{|\xi| \geq  \delta^{-1} }\chi_3 (\xi) a_s(\xi)  \int_{\R^2_{y}} g(y)  e^{iN  \Psi(y)  }    \, dy \, d\xi  \Big|\\
 & \lesssim  N^{-n} \int_{    |\xi| \geq  \delta^{-1}  } |\xi|^{-s } \int_{y\in \Supp g }  |[L_y^*]^n (g) |        \, dy \, d\xi \\
& \lesssim N^{-n} \int_{    |\xi| \geq  \delta^{-1}  } |\xi|^{-s -2}         \, d\xi \\
& \lesssim N^{-s-2}.
\end{align*}

\noindent
\textbf{Estimate of $I_{N,1}$}

For
\begin{equation}
I_{N,1} = \int_{|\xi| \leq 2\delta }\chi_1(\xi) a_s(\xi)e^{ iN x\cdot \xi }  \int_{\R^2_{y}}  {g}  e^{iN  \phi(y)  } e^{-iN y\cdot \xi  }   \, dy \, d\xi ,
\end{equation}
we still have $ |\nabla_{y,\xi} \Psi | \geq |\xi - \nabla \phi | \geq \delta$, so we want to argue similarly as above. The main difficulty is that 
$$
\int_{|\xi| \le 2\delta}a_{s}(\xi)\,d\xi=\infty
$$
when $s\ge 2$. Therefore, we need to gain some factor of order $|\xi|^s$. Using polar coordinate $y=(r\cos \alpha, r\sin \alpha)$, there hold
\begin{equation}
    \begin{aligned}
        &\int_{\R^2_{y}}  {g}  e^{iN  \phi(y)  } e^{-iN y\cdot \xi  }   \, dy\\
        &=  \int_0^{\infty}   \left(\int_{-\pi}^{\pi} e^{iN  \alpha  } e^{-irN (\xi_1\cos \alpha+\xi_2\sin \alpha)  }  \,d\alpha \right) {rg(r)}e^{iNh(r)} \, dr\\ 
        &=:\int_0^{\infty}   J_{N,N}(r,\xi) {rg(r)}e^{iNh(r)} \, dr,
    \end{aligned}
\end{equation}
where
$$
J_{N,M}=\int_{-\pi}^{\pi} e^{iN  \alpha  } e^{-irM (\xi_1\cos \alpha+\xi_2\sin \alpha)  }  \,d\alpha.
$$
Integrating by part, we obtain
\begin{equation}
    \begin{aligned}
        J_{N,M}
        =& \frac{1}{iN}\int_{-\pi}^{\pi}  e^{-irM (\xi_1\cos \alpha+\xi_2\sin \alpha)  }  \,de^{iN  \alpha  } \\
        =& \frac{Mr}{N} \int_{-\pi}^{\pi}\left( -\xi_1\sin \alpha+\xi_2\cos \alpha \right)e^{iN\alpha}e^{-irM (\xi_1\cos \alpha+\xi_2\sin \alpha)  }  \,d\alpha.
    \end{aligned}
\end{equation}
Using the identity $\sin \alpha =\frac{e^{i\alpha}-e^{-i\alpha}}{2i}$ and $\cos \alpha=\frac{e^{i\alpha}+e^{-i\alpha}}{2}$, we get

\begin{equation}
    \begin{aligned}
        J_{N,M}
        =& \frac{-rM\xi_1}{2Ni} \int_{-\pi}^{\pi}e^{i(N+1)\alpha}e^{-irM (\xi_1\cos \alpha+\xi_2\sin \alpha)  }  \,d\alpha \\
        &+\frac{rM\xi_1}{2Ni} \int_{-\pi}^{\pi}e^{i(N-1)\alpha}e^{-irM (\xi_1\cos \alpha+\xi_2\sin \alpha)  }  \,d\alpha \\
        &+\frac{rM\xi_2}{2N} \int_{-\pi}^{\pi}e^{i(N+1)\alpha}e^{-irM (\xi_1\cos \alpha+\xi_2\sin \alpha)  }  \,d\alpha \\
        &+\frac{rM\xi_2}{2N} \int_{-\pi}^{\pi}e^{i(N-1)\alpha}e^{-irM (\xi_1\cos \alpha+\xi_2\sin \alpha)  }  \,d\alpha\\
        =&\frac{-rM\xi_1}{2Ni}J_{N+1,M}+\frac{rM\xi_1}{2Ni}J_{N-1,M}+\frac{-rM\xi_2}{2N}J_{N+1,M}+\frac{rM\xi_2}{2N}J_{N-1,M}.
    \end{aligned}
\end{equation}
Set $m=\lfloor s \rfloor +1$, which is much smaller than $N$. Repeating $m$ times, we have
\begin{equation}
    \begin{aligned}
        J_{N,N}=\sum_{j=N-m}^{N+m}r^mb_{j,m}(\xi)J_{j,N}(\xi),
    \end{aligned}
\end{equation}
where
$$
|\partial_{\xi}^{k} b_{j,m}(\xi)| \lesssim_{m,k} |\xi|^{m-k}  
$$
provided that $k \le m$. Therefore,
\begin{equation}\label{eq in1}
    \begin{aligned}
        I_{N,1}=&\sum_{j=N-m}^{N+m}\int_{|\xi| \le 2\delta}\chi_1(\xi)a_{s}(\xi)b_{j,m}(\xi)\int_{\R_y^2}|y|^mg(y)e^{i(j-N)\alpha}e^{iN\Psi(y)}\,dy \,d\xi. 
    \end{aligned}
\end{equation}
Note that $|j-N|\le m \le s+2$ and 
$$
\int_{|\xi|\le 2\delta}\chi_1(\xi)a_{s}(\xi)|b_{j,m}(\xi)|\,d\xi \lesssim_{\delta,s}1,
$$
via a similar argument as that used in the estimate of $I_{N,3}$, we obtain
$$
|I_{N,1}| \lesssim N^{-s-2}.
$$

\noindent
\textbf{Estimate of $I_{N,2}$}

In this regime, 
$$
I_{N,2} = \int_{\frac\delta2 \leq |\xi| \leq 2\delta^{-1}  } \int_{\R^2_{y}} g(y) \chi_2(\xi) a_s(\xi) e^{iN \Psi }    \, dy \, d\xi
$$
and we first solve for stationary points $\nabla_{y,\xi} \Psi = 0$. 

There is a unique stationary point at $(y_0, \xi_0) \in \R^2_{y} \times \R^2_{\xi} $ with $y_0 = x $ and $\xi_0 =\nabla \phi(x)$ to the equation
\begin{equation}
\partial_\xi \Psi = x-y \qquad \partial_y \Psi = -\xi + \nabla \phi(y).
\end{equation}

At the stationary point  $(y_0, \xi_0)$, the Hessian 
\begin{equation}
D^2\Psi (y_0, \xi_0) =  \left(
\begin{tabular}{ll}
$D^2  \phi(x)$ & $- \mathrm{Id}$ \\
$- \mathrm{Id}$ & 0
\end{tabular} \right)
\end{equation}
satisfies $ \det(D^2\Psi  ) =  1$ with $ \mathrm{sgn}(D^2\Psi ) = 0$, where $\mathrm{sgn}(D^2\Psi)$ denotes the signature of the Hessian matrix $D^2\Psi$, i.e., the difference between the number of positive and negative eigenvalues.

Thus by the stationary phase approximation (see for instance \cite[Chapter 3.5]{MR2952218}),
\begin{equation}
\Lambda^{-s} \theta(x)   = N^{2-s } \left( g(x) a_s( \nabla \phi(x))  e^{i N \phi(x)} (2\pi/ N )^2 + O(N^{-3} ) \right) .
\end{equation}
Here the $O(N^{-3} )$ term depends on higher order derivatives of the amplitude function $ (y,\xi) \mapsto  g(y)  \chi_2(\xi) a_s( \nabla\phi(y))$ which are   bounded in  $   \R^2_y\times \R^2_\xi$ due to support constraints.

Using that $\nabla \phi (x) = h'(r) \mathbf{e}_r + \frac1r \mathbf{e}_\theta $ we get
\begin{equation}
\Lambda^{-s} \theta(x)   = c_s N^{ -s } \left( \frac{g(r) }{(|\frac1r|^2 + |h'(r)|^2 )^{\frac{s}{2}}}   e^{i N \phi(x)}  \right) + O(N^{-1-s} ) .
\end{equation}
\paragraph{Estimates when $|x|\ge 2 \gamma$.} 

In this regime, we need to show suitable decay at infinity. 

For \eqref{eq Lambda out}, note that
$$
\frac{\theta(x) }{\left[ \left( \frac{N}{r} \right)^2 + N^2 h'(r)^2 \right]^{\frac{s}{2}}}\equiv 0
$$
for any $|x| \ge 2\gamma$ since $\operatorname{supp} g \subset [\frac{1}{\gamma},\gamma]$, it remains to estimate $\Lambda^{-s}\theta$. Recall that
\begin{align*}
 \Lambda^{-s} \theta(x) &  = N^{2-s}I_N(x)=N^{2-s}\sum_{j=1}^3I_{N,j}(x)
 \end{align*}
 where
 $$I_{N,j}(x)=\int_{\R^2_{\xi}} \int_{\R^2_{y}} g(y) e^{iN \phi(y)} e^{iN(x-y)\cdot \xi} a_s(\xi)\chi_j(\xi) \, dy \, d\xi$$
\noindent
\textbf{Estimate of $I_{N,2}$}
Consider the operator
$$
L_{\xi}:=\frac{x-y}{i|x-y|^2}\cdot \nabla_{\xi}
$$
so that
$$
\frac{L_{\xi}}{N}\left( e^{iN\Psi} \right)=e^{iN\Psi}.
$$
Thus, for any $n>0$, there hold
$$
I_{N,2}=N^{-n}\int_{\R^2_{\xi}} \int_{\R^2_{y}} g(y) e^{iN \Psi(y)}  \left[L^*_{\xi} \right]^n\left(a_s(\xi)\chi_2(\xi)\right) \, dy \, d\xi.
$$
Note that $|y|\le \gamma$ whenever $y \in \operatorname{supp} g$, which implies $|x-y|\sim |x|$ for any $|x|\ge 2\gamma$. Thus, for any $n>0$, it follows that
$$
|I_{N,2}(x)|\lesssim N^{-n}r^{-n}.
$$

\noindent
\textbf{Estimate of $I_{N,3}$}

Arguing similarly as above, we get
\begin{equation*}
I_{N,3}   =   N^{-n} \int_{\R_{\xi}^2}\int_{\R_y^2} g(y)e^{iN\Psi}[L_{\xi}^*]^n\left(\chi_3 (\xi) a_s(\xi)           \right)\, dy \, d\xi.
\end{equation*}
 Note that
 $$
 \left| [L_{\xi}^*]^n(\chi_3(\xi)a_s(\xi)) \right| \lesssim \xi^{-s-n}r^{-n},
 $$
 taken $n=\lfloor s \rfloor +3$, we obtain
\begin{align*}
     |I_{N,3}| \lesssim& N^{-n}\int_{\R_y^2}\int_{\R_{\xi}^2} |g(y)|
     \left| [L_{\xi}^*]^n(\chi_3(\xi)a_s(\xi)) \right|\,d\xi \,dy\\
     \lesssim& N^{-n}r^{-n} \\
     \lesssim& N^{-1-s}r^{-3}.
\end{align*}

\noindent
\textbf{Estimate of $I_{N,1}$}
Using the operator $L_{\xi}$, it follows from \eqref{eq in1} that
\begin{align*}
        I_{N,1}=N^{-n}&\sum_{j=N-m}^{N+m}\int_{\R_{\xi}^2}\int_{\R_y^2} [L_{\xi}^*]^n\left(\chi_1(\xi)a_{s}(\xi)b_{j,m}(\xi)\right)|y|^mg(y)e^{i(j-N)\alpha}e^{iN\Psi(y)}\,dy \,d\xi .
    \end{align*}
    Take $n=\lfloor s \rfloor +3$ and $m=n+\lfloor s\rfloor+2$, then 
    $$
    \left| [L_{\xi}^*]^n\left(\chi_1(\xi)a_{s}(\xi)b_{j,m}(\xi)\right) \right| \lesssim r^{-n}|\xi|^{m-n-s} \lesssim r^{-3}|\xi|
    $$
    when $|\xi| \in \operatorname{supp} \chi_1$. Therefore, it follows directly that
    $$
    |I_{N,1}| \lesssim N^{-1-s}r^{-3}.
    $$
\end{proof}

\subsection*{Funding}
XL is supported by National Natural Science Foundation of China [grant numbers 12421001, 12288201].

\subsection*{Data Availability Statement}

Data sharing is not applicable to this article as no datasets were generated or analyzed during the current study.

\subsection*{Conflict of Interest}

The authors declare that they have no conflict of interest.

\bibliography{reference}  % 假设references.bib含CC条目

\end{document}